\documentclass{amsart}

\usepackage{color}
\makeatletter
 \def\LaTeX{\leavevmode L\raise.42ex
   \hbox{\kern-.3em\size{\sf@size}{0pt}\selectfont A}\kern-.15em\TeX}
\makeatother

\newcommand{\BibTeX}{{\rm B\kern-.05em{\sc
i\kern-.025emb}\kern-.08em\TeX}}
\newtheorem{corollary}{Corollary}[section]
\newtheorem{theorem}{Theorem}[section]
\newtheorem{lemma}[theorem]{Lemma}

\newtheorem{remark}[theorem]{Remark}
\newtheorem{assumption}[theorem]{Assumption}
\newtheorem{definition}[theorem]{Definition}

\newtheorem{counterexample}{Counterexample}

\numberwithin{equation}{section}

\newcommand{\sinc}{\mathrm{sinc}}

\def\EB{{\bf E}}

\begin{document}

\title[Reconstructing solutions of abstract differential equations ]{Reconstruction of Solutions to Abstract Schrödinger, Diffusion, and Wave-Type Equations from Snapshots.}

\author{Isaac Z. Pesenson }
\address{ Department of  Mathematics, Temple University,
 Philadelphia,
PA 19122; \\pesenson@temple.edu  }


\begin{abstract}

It is shown that if one knows "sufficiently many" snapshots (samples)  from an individual solution of a Schrödinger-type equation-which is bandlimited in an appropriate sense-then the solution can be reconstructed by a simple formula similar to the Shannon formula for bandlimited signals. Furthermore, results demonstrate that by using snapshots of bandlimited solutions of a Schrödinger-type equation, one can reconstruct bandlimited solutions of some relevant diffusion- and wave-type equations. Interestingly, the reconstruction formulas are essentially independent of the operators defining the equations.

\end{abstract}

\maketitle

\section{Introduction}

It is known that for the so-called uniformly correct Cauchy problems for abstract differential equations of the Schrödinger, diffusion, or wave type, the propagator can be expressed in terms of a one-parameter group or semigroup generated by a relevant operator. However, in most cases, the explicit formula for such a semigroup is unknown.

In this paper, it is shown that if one knows "sufficiently many" snapshots (the spatial profiles of  solutions   at  fixed instants in time \cite{BK}, \cite{LS})   from an individual solution of a Schrödinger-type equation that belongs to a Bernstein subspace (i.e., a bandlimited trajectory of a relevant one-parameter semigroup), then the solution can be perfectly reconstructed using a "simple" formula. This formula is similar to the Shannon formula for bandlimited signals, involving only the \(\mathrm{sinc}\) function.

Moreover, our results demonstrate that by using snapshots  of bandlimited solutions of a Schrödinger-type equation, one can reconstruct bandlimited solutions of some relevant diffusion- and wave-type equations. We show that the subsets of such bandlimited solutions are dense in the set of all solutions. The reconstruction formulas are essentially independent of the operators defining the equations. 

The paper is organized as follows.
In the first half of Section \ref{prelim}  we introduce sets of Bernstein (bandlimited) vectors ${\bf B}_{\sigma}(D),\>\sigma>0,$ in a Banach space ${\bf E}$ associated with a generator $D$ of a $C_{0}$ one-parameter group of isometries $G_{D}$. In particular, we show that when $f\in {\bf B}_{\sigma}(D)$, the trajectory $G_{D}(t)f, \>t\in \mathbb{R},$ is a bounded, vector-valued function on $\mathbb{R}$ which has extension to $\mathbb{C}$ as an entire function of  exponential type $\sigma$. We prove in the Appendix that the set $\bigcup_{\sigma>0} {\bf B}_{\sigma}(D)$ is dense in ${\bf E}$. 
In the second half of this section we prove sampling/interpolation formulas for such trajectories $G_{D}(t)f, \>t\in \mathbb{R},$ where $f\in {\bf B}_{\sigma}(D)$.

In subsection \ref{Schrod-eq}  we formulate our results on sampling and interpolation for solutions of Cauchy problems for abstract differential equations of the Schr\"{o}dinger type which are defined by generators of one parameter groups.
The objective of subsection \ref{heat-eq} is to obtain reconstruction  formulas for Bernstein solutions of abstract Cauchy problems for the diffusion-type equations which are typically determined by operators which generate one parameter semigroups. We demonstrate that unlike the situation with generators of one parameter groups, one cannot expect that the corresponding set $\bigcup_{\sigma>0} {\bf B}_{\sigma}(A)$ for a semigroup generator $A$ to be  dense in ${\bf E}$.

In other words,  if $A$ generates a semigroup and one wants  the set $\bigcup_{\sigma>0}{\bf B}_{\sigma}(A)$ to be dense in ${\bf E}$, then some additional restrictions on the operator $A$ must be imposed.  Our additional conditions are the following: 

\textit{We consider an operator  $A$ which generates in ${\bf E}$ a contraction $C_{0}$-semigroup $S_{A}(t), \>\|S_{A}(t)\|\leq 1,\>0\leq t<\infty,$  and is  such that for some $\phi\in [ -\pi, \pi)$ the operator $e^{i\phi}A$ generates  an isometry $C_{0}$-group of operators $G_{e^{i\phi}A}(t),\>t\in \mathbb{R}$.  }

This assumption allows us to explore the obvious property  
$$
{\bf B}_{\sigma}(A)={\bf B}_{\sigma}(e^{i\phi} A),
$$
and, moreover,  we are able to reduce questions about reconstruction of the trajectories of the semigroup $S_{A}$ to  questions about the reconstruction of the trajectories of the group $G_{e^{i\phi}A}$.
In fact, 
the framework we consider is motivated by the following classical example.
It is known that the heat equation 
$$
 \frac{d}{d\tau}h(\tau, x)=\Delta h(\tau, x),\>\>\>\tau\in \mathbb{R}_{+}, \>x\in \mathbb{R}^{d}, 
$$
becomes  the Schr\"{o}dinger equation 
$$
 \frac{d}{dt}\varphi(t, x)=i\Delta\varphi(t, x),\>\>\>i^{2}=-1, \>\>\>t\in \mathbb{R}, \>x\in \mathbb{R}^{d}, 
$$
where $\Delta$ is the Laplacian in $L_{2}(\mathbb{R}^{d})$,
under the transformation 
\begin{equation}\label{trans}
t=-i\tau,
\end {equation}
which maps "real" time to "imaginary" time and is known as the Wick rotation (\cite{PS}, chapter 6). The reverse   transformation $\tau=it$,
is called the reverse Wick rotation. 
We generalize this  classical situation  by considering the operator $A$ instead of $\Delta $ and $e^{i\phi}A$ instead of $i\Delta$. Our transformation 
\begin{equation}\label{our-trans}
t=e^{i\phi}\tau, 
\end{equation}
 is a generalization of the Wick transformation, which corresponds to the case where $\phi=\pi/2$. 
Our sampling results are new even in this classical case, in which our spaces ${\bf B}_{\sigma}(\Delta)$ coincide with the regular Paley-Wiener spaces $PW_{\sigma}$ of all functions in $L_{2}(\mathbb{R}^{d})$ whose Fourier transform  have support  in the cube $[-\sigma, \sigma]^{d}$.

In subsection \ref{wave-eq-1} we  consider a Cauchy problem for the wave equation of the form 
$$
\frac{d^{2}}{ds^{2}}w(s)=D^{2}w(s),\>\>s\in \mathbb{R}, \>\>w: \mathbb{R}\mapsto \mathcal{D}(D^{2}),
$$
where $D$ generates a one-parameter $C_{0}$-group $G_{D}$ of isometries and $\mathcal{D}(D^{2})$ is the domain of $D^{2}$. 
In this case we give explicit formulas for solutions of the corresponding Cauchy problems with bandlimited data in terms of snapshots  of 
trajectories of $G_{D}$. The paper also contains several formulas for weak reconstruction of bandlimited solutions of non-stationary abstract differential equations. Section \ref{classicalexamp} contains classical examples.  In Appendix we prove the fact that if $D$ generates a $C_{0}$-group of isometries in a Banach space ${\bf E}$, then set $\cup_{\sigma}{\bf B}_{\sigma}(D),\>\sigma>0,$ is dense in ${\bf E}$.
For the necessary background on abstract non-stationary equations in Banach spaces, we recommend the monographs \cite{Birman, EN, K}.

\section{Preliminaries}\label{prelim}

\subsection{One-parameter groups of operators in Banach spaces and subspaces of abstract bandlimited (Bernstein)   vectors}

The algebra of all bounded operators on a Banach space  $\EB$ with a norm $\|\cdot\|$ will be denoted as $\mathcal{B}(\EB)$. 
We say that $D$ generates a one-parameter $C_{0}$-group group $G_{D}$  of isometries  if
$$
G_{D}: \mathbb{R}\mapsto \mathcal{B}(\EB)
$$ 
\begin{enumerate}
\item $G_{D}(t_{1}+t_{2})=G_{D}(t_{1})G_{D}(t_{2}),\>\>t_{1}, t_{2}\in \mathbb{R},$
\item $G_{D}(0)=I,\>\>(I - is \>the \>identity \>operator),$
\item  $\|G_{D}(t)\|=1,\>\>t\in \mathbb{R}$,
\item $\lim_{t\rightarrow 0}\|G_{D}(t)-f\|=0,\>\>f\in \EB$.
\end{enumerate}

Such an operator is necessarily  has a dense domain and is closed. The notations $\mathcal{D}^{k}(D)$ will be used for  the domain of $D^{k}$, and notation $\mathcal{D}^{\infty}(D)$ for $\bigcap_{k\in \mathbb{N}}\mathcal{D}^{k}(D)$. If $f\in \mathcal{D}(D)$ then 
$$
\frac{d}{dt}G_{D}(t)f=\lim_{t\rightarrow 0}t^{-1}(G_{D}(t)f-f)=DG_{D}(t)f=G_{D}(t)Df,\>\>f\in \mathcal{D}(D),
$$
where convergence holds in the norm of $\EB$. The following definition was introduced and systematically used in papers \cite{Pes88}-\cite{Pes15}.

\begin{definition}\label{Bernstein space}
The Bernstein subspace
 $\mathbf{B}_{\sigma}(D), \>\sigma> 0,$ is defined as  a set of all vectors $f$ in $\mathcal{D}^{\infty}(D)$  for which the  Bernstein-type  inequalities hold
\begin{equation}\label{group-exponent}
\|D^{k}f\|\leq \sigma^{k}\|f\|, \>\>k\in \mathbb{N}.
\end{equation}
\end{definition}

It is obvious that every $\mathbf{B}_{\sigma}(D), \>\sigma> 0$, is invariant under $D$ and $G_{D}$.
\begin{theorem}\label{exp-solution}
Suppose that 
 $D$ generates a one-parameter uniformly bounded $C_{0}$-group $G_{D}$,  and  a vector $f$  belongs to a space ${\bf B}_{\sigma}(D)$. Then the following holds true
 \begin{enumerate}

 \item The trajectory $G_{D}(t)f$ has extension to the complex plain as an abstract valued entire function

\begin{equation}\label{ext}
G_{D}(z)f=\sum_{k=0}^{\infty}\frac{z^{k}}{k!}D^{k}f=e^{zD}f,\>\>z\in \mathbb{C},
\end{equation}
where the series converges uniformly on compact subsets of $\mathbb{C}$.

\item The extended trajectory $G_{D}(z)f$ has exponential type, i.e.

\begin{equation}\label{Im}
\|G_{D}(z)f\|\leq \|f\|e^{\sigma\left|\Im\> z\right|}.
\end{equation}

\end{enumerate}

\end{theorem}

\begin{proof}
Convergence of the series on bounded subsets of $\mathbb{R}$ is obvious since on every interval $[0, \>b]$ one has for all $t\in [0, \>b]$:
$$
\left\| \sum_{k=0}^{\infty}\frac{t^{k}}{k!}D^{k}f\right\|\leq \sum_{k=0}^{\infty}\frac{(b\sigma)^{k}}{k!}\|f\|=e^{b\sigma}\|f\|.
$$ 
To prove the equality, consider a Taylor series for the function $F(t)=\langle G_{D}(t)f, f^{*}\rangle$ where $f^{*}$ belongs to the dual $\EB^{*}$
\begin{equation}\label{Taylor1}
F_{D}(t)=\sum_{k=0}^{n-1}\frac{t^{k}}{k!}F^{(k)}(0)+\frac{1}{(n-1)!}\int_{0}^{t}(t-s)^{n-1}F^{(n)}(s)ds.
\end{equation}
Since
$$
F^{(k)}(0)=\langle D^{k}f, f^{*}\rangle,\>\>\>\>
F^{(n)}(s)=\langle G_{D}(s)D^{n}f, f^{*}\rangle,
$$
and the integral 
$$
\frac{1}{(n-1)!}\int_{0}^{t}(t-s)^{n-1}G_{D}(s)D^{n}fds
$$
converges in $\EB$,
one obtains for every $f^{*}\in \EB^{*}$ and every $t\in \mathbb{R}$
$$
\left\langle G_{D}(t)f-\sum_{k=0}^{n-1}\frac{t^{k}}{k!}D^{k}f, f^{*}\right\rangle=
\left\langle \frac{1}{(n-1)!}\int_{0}^{t}(t-s)^{n-1}G_{D}(s)D^{n}fds,\>f^{*}\right\rangle.
$$
 It implies the formula
 $$
 G_{D}(t)f-\sum_{k=0}^{n-1}\frac{t^{k}}{k!}D^{k}f=\frac{1}{(n-1)!}\int_{0}^{t}(t-s)^{n-1}G_{D}(s)D^{n}fds.
 $$
 Because  the norm $\|\cdot\|$ of the integral goes  to zero uniformly on compact sets in $\mathbb{R}$ when $n$ goes to infinity,  we obtain that
 
\begin{equation}\label{ext to Z}
 G_{D}(t)f=\sum_{k=0}^{\infty}\frac{t^{k}}{k!}D^{k}f,\>\>\>f\in {\bf B}_{\sigma}(D),
\end{equation}
 where the series converges uniformly on compact subsets of $\mathbb{R}$.
One can use this series to extend $G_{D}(t)f$ to the complex plain  $\mathbb{C}$ as an entire abstract function. 
Since the group property still holds for this extension, one obtains the following chain  of inequalities for a $z=x+i y\in \mathbb{C}$:
$$
\|G_{D}(z)f\|=\|G_{D}(x+i y)f\|=\|G_{D}(x)G_{D}(i y)f\|\leq \|G_{D}(iy)f\|=
$$
$$
\left\|\sum_{k}\frac{(iy)^{k}}{k!}D^{k}f\right \|\leq \|f\|e^{\sigma |y|},\>\>\>f\in {\bf B}_{\sigma}(D).
$$
Lemma is proven.

\end{proof}

\begin{corollary}
 If $D$ generates a one-parameter $C_{0}$-group of isometries, then the following holds true.
 
 \begin {enumerate}
 
 \item Every $\mathbf{B}_{\sigma}(D)$ is a  linear subspace of ${\bf E}$. 
 
 \item Every $\mathbf{B}_{\sigma}(D)$ is a closed subspace of ${\bf E}$.

 \end{enumerate}
 
\end{corollary}

\begin{proof}
The first item is a direct consequence of the previous Lemma.
To prove the second item we consider the sequence $\{f_{j}\}\in \mathbf{B}_{\sigma}(D)$ which converges in ${\bf E}$ to $f$.  Since the sequence $\{f_{j}\}$ is fundamental in ${\bf E}$,  the Bernstein inequality implies that $\{D^{k}f_{j}\},\>\>k\in \mathbb{N},$ is fundamental. Since $D^{k},\>\>k\in \mathbb{N},$ is a closed operator it implies that $f$ belongs to the domain of $D^{k},\>\>k\in \mathbb{N},$ and the Bernstein inequality 
$$
\|D^{k}f\|\leq \sigma^{k}\|f\|
$$
holds for any $k\in \mathbb{N}$.
  Hence the corollary is proved.

\end{proof}

\begin{definition}
 If $D$ generates a one-parameter $C_{0}$-group of isometries $G_{D}$, we say that a vector $f\in \EB$ belongs to ${\bf WB_{\sigma}}(D),\>\sigma>0,$ if for every functional $f^{*}\in \EB^{*}$ the function $\langle G_{D}(t)f, f^{*}\rangle$  belongs to the space $B_{\sigma}^{\infty}(\mathbb{R})$ of entire functions of  exponential type $\sigma$ bounded on the real line.
\end{definition}

\begin{theorem}
 If $D$ generates a one parameter group of isometries of class $C_{0}$ then for every $\sigma>0$ the subspaces ${\bf B}_{\sigma}(D)$ and ${\bf WB}_{\sigma}(D)$ coincide.
\end{theorem}
\begin{proof}
If $f\in {\bf WB}_{\sigma}(D)$ then for every $f^{*}\in \EB^{*}$ the function $\langle G_{D}(t)f, f^{*}\rangle$ satisfies the classical Bernstein inequality  \cite{Akh}, i.e.
$$
\sup_{t\in \mathbb{R}}\left| \left(\frac{d}{dt}\right)^{k}\langle G_{D}(t)f, f^{*}\rangle\right 
|\leq \sigma^{k}\sup_{t\in \mathbb{R}}\left|\langle G_{D}(t)f, f^{*}\rangle\right|,
$$
which implies
$$
\sup_{t\in \mathbb{R}}|\langle G_{D}(t)D^{k}f, f^{*}\rangle|\leq \sigma^{k} \|f^{*}\|\|f\|.
$$
From here, for the functional for which $\|f^{*}\|=1$, and $\langle G_{D}(t)D^{k}f, f^{*}\rangle=\|D^{k}f\|$ we obtain the Bernstein inequalities 
$$
\|D^{k}f\|\leq \sigma^{k}\|f\|,\>\>k\in \mathbb{N},
$$
which shows that $f\in {\bf B}_{\sigma}(D)$.
Conversely, if $f\in {\bf B}_{\sigma}(D)$ then according to Lemma \ref{exp-solution} we have 
$
G_{D}(t)f
$
has extension  $G_{D}(z)f,\>z\in \mathbb{C}$  as abstract function of exponential type $\sigma$.
 It implies that for any $f^{*}\in \EB^{*}$ the scalar valued function $ \langle G_{D}(t)f, f^{*}\rangle$ can be extended as entire function to $\mathbb{C}$, and the next estimate holds 
$$
\left| \langle G_{D}(z)f, f^{*}\rangle\right|\leq \|f\|\|f^{*}\| \sum
^{\infty}_{k=0}|z|^{k}\sigma^{k}/k!= \|f\|\|f^{*}\|e^{|z|\sigma}.
$$
Thus the function $ \langle G_{D}(z)f, f^{*}\rangle$ is bounded on $\mathbb{R}$, has extension to $\mathbb{C}$ as an entire function, and has exponential type $\sigma$, which means that $f\in  {\bf WB}_{\sigma}(D)$. Theorem is proven.

\end{proof}

The following important fact was proven in   \cite{Pes14}.

\begin{theorem}\label{key-Bern-th}
  If $D$ generates a one-parameter group of isometries of class $C_{0}$ in a Banach space $\mathbf{E}$ then the set $\bigcup_{\sigma>0}{\bf B}_{\sigma}(D),$ is dense in $\mathbf{E}$.

\end{theorem}

The proof of this theorem is provided in the Appendix for completeness.

\subsection{Valiron-Tschakalov sampling and interpolation formula for bandlimited trajectories}\label
{ sampling group trajectories}

We assume that $D$ generates one-parameter strongly continuous  group of isometries  $G_{D}(t), \>\>t\in \mathbb{R},$ in a Banach space ${\bf E}$. In this section we prove explicit formulas for a  trajectory $G_{D}(t)f$ with $f\in \mathbf{B}_{\sigma}(D)$ in terms of a countable number of equally spaced samples. 
Let us remind that  the function $ \sinc$ is defined for $x\in \mathbb{R}$ as follows
{
$$
  \sinc\> x=\begin{cases}
                \frac{\sin \pi x }{\pi x}&,  \, \text {x $\neq 0$}\\
                1                    & ,\, \text {x$=0$}.
                \end{cases}
                $$}

The next theorem is given in \cite{BFHSS}.
\begin{theorem}\label{exact sampling}
If $f$ belongs to the space $B_{\sigma}^{2}(\mathbb{R}),$ of all  functions in $L_{2}(\mathbb{R})$  which have extension to $\mathbb{C}$ as entire functions of exponential type $\sigma$, 
then for $z\in \mathbb{C}$ one has
\begin{equation}\label{S20}
f(z)=\sum_{k\in \mathbb{Z}}  f\left(\frac{k\pi}{\sigma}\right) \sinc\left(\frac{\sigma}{\pi}z-k\right),
\end{equation}
where the series absolutely and uniformly converges on all strips of bounded width parallel to the real axis.
\end{theorem}

\begin{remark}
If $\frac{\sigma}{\pi}z\in \mathbb{Z}\setminus\{0\}$ then every term containing $\sinc \left(\frac{\sigma}{\pi}z-k\right)$
in the series (\ref{S20})
 is zero,  except  when $\frac{\sigma}{\pi}z=k$ (or $z=\frac{k\pi}{\sigma}$).
 In this case, the corresponding non-trivial term is  $f(k\pi/\sigma)$.
 In other words, the left- and  right-hand sides of the formula (\ref{S20}) are identical. It is why these formulas are known as interpolation formulas.
\end{remark}

We are going to use Theorem \ref{exact sampling} to obtain  a generalization of the Valiron-Tschakaloff  \cite{BFHSS} sampling and interpolation theorem.

\begin{lemma}
The following identity holds
 \begin{equation}\label{identity}
\frac{\sinc\>z-1}{z}=-\sum_{k\in \mathbb{Z}\setminus \{0\}}\frac{\sinc(z-k)}{k}.
\end{equation}
\end{lemma}

\begin{proof}
Consider the following function  
$$
\psi(z)=\begin{cases}
                \frac{1-\sinc\> z}{z}&,\, \text {z $\neq 0$}\\
               0                    &,\, \text {z$=0$}.
                \end{cases}
                $$
                One has for $k\in \mathbb{Z}$, 
                
                $$
\psi(k)=\begin{cases}
                \frac{1}{k}&,\, \text {k $\neq 0$}\\
               0                    &,\, \text {k$=0$}.
                \end{cases}
                $$
                The  function $\psi$ belongs to $B_{\pi}^{2}(\mathbb{R})$ and we can apply to it  Theorem \ref{exact sampling} to obtain

\begin{equation}\label{intermediate}
\psi(z)=\sum_{k\in \mathbb{Z}\setminus \{0\}}\psi(k) \sinc(z-k),
\end{equation}
which is exactly (\ref{identity}). Lemma is proven.

\end{proof}

\begin{theorem}\label{V-T}
If $D$ generates a one parameter $C_{0}$-group $G_{D}(t)$ of isometries,  and $f\in \mathbf{B}_{\sigma}(D)$, then the following sampling formula holds 
\begin{equation}\label{s100}
G_{D}(z)f=\sinc\left(\frac{\sigma z}{\pi}\right)f+
$$
$$
z\>\sinc\left(\frac{\sigma z}{\pi}\right)Df +
\sum_{k\in \mathbb{Z}\setminus\{0\}}   \left(\frac{\sigma z}{k\pi}\right)     \sinc\left(\frac{\sigma z}{\pi}-k\right)G_{D}\left(\frac{k\pi}{\sigma}\right)f ,
\end{equation}
where the  series converges in the norm of ${\bf E}$ uniformly on  all strips of bounded width parallel to the real axis, and in particular, on compact subsets of $\mathbb{C}$.
\end{theorem}

       \begin{proof}
       
         If $f\in \mathbf{B}_{\sigma}(D)$ then for any $g^{*}\in {\bf E}^{*}$ the function $F(t)=\left<G_{D}(t)f,\>g^{*}\right>$ belongs to $B_{\sigma}^{\infty}(\mathbb{R})$.
We consider $\Phi\in B_{\sigma}^{2}( \mathbb{R}),$ which is defined as follows.
If $t\neq 0$ then 
\begin{equation}\label{F}
\Phi(t)=\frac{F(t)-F(0)}{t}=\left<\frac{G_{D}(t)f-f}{t},\>g^{*}\right>,
\end{equation}
and if $t=0$ then
$
\Phi(0)=\frac{d}{dt}F(t)|_{t=0}=\left<Df,\>g^{*}\right>.
$
According to Theorem \ref{exact sampling} one has the equality
$$
\Phi(z)=\left<Df,\>g^{*}\right>+\sum_{k\in \mathbb{Z}\setminus\{0\}}\Phi\left(\frac{k\pi}{\sigma}\right)\> \sinc\left(\frac{\sigma }{\pi}z-k\right),
$$
or 
$$
\left<    \frac{G_{D}(z)f-f}{z},\>g^{*}    \right>=
$$
$$
\left<Df,\>g^{*}\right>+\sum_{k\in \mathbb{Z}\setminus\{0\}}\left<\frac{ G_{D}\left(\frac{k\pi}{\sigma}\right)f-f}{\frac{k\pi}{\sigma}},\>g^{*}\right>\> \sinc\left(\frac{\sigma z}{\pi}-k\right),
$$
where the series absolutely and uniformly converges on all strips of bounded width parallel to the real axis, and in particular, on compact subsets of $\mathbb{C}$.
We note that the series
$$
\sum_{k\in \mathbb{Z}\setminus\{0\}}
\frac{G_{D}\left(\frac{k\pi}{\sigma}\right)f-f}{\frac{k\pi}{\sigma}} 
\sinc\left(\frac{\sigma z}{\pi}-k\right), 
$$
converges in the norm of ${\bf E}$ uniformly on compact subsets of $\mathbb{C}$.
It leads to the following equality on compact subsets of $\mathbb{C}$ for any $g^{*}\in {\bf E}^{*}$ 
$$
\left<    \frac{G_{D}(z)f-f}{z},\>g^{*}    \right>=
$$
$$
\left<Df+\sum_{k\in \mathbb{Z}\setminus\{0\}}
\frac{G_{D}\left(\frac{k\pi}{\sigma}\right)f-f}{\frac{k\pi}{\sigma}} 
\sinc\left(\frac{\sigma z}{\pi}-k\right),\>g^{*}\right>,\>\>\>z\in \mathbb{C},
$$
and if $z= 0$ it gives the identity
$$
\left<Df,\>g^{*}\right>=\left<Df,\>g^{*}\right>.
$$
Thus one obtains
 for every $z\in \mathbb{C}$ 
 \begin{equation}\label{samplingformulabeforeVT}
G_{D}(z)f=f+zDf \sinc\left(\frac{\sigma z}{\pi}\right)+z\sum_{ k\in \mathbb{Z}\setminus\{0\}}
\frac{G_{D}\left(\frac{k\pi}{\sigma}\right)f-f}{\frac{k\pi}{\sigma}} 
\sinc\left(\frac{\sigma z}{\pi}-k\right).
\end{equation}
By using the identity (\ref{identity}
one obtains
$$
\frac{\sigma z}{\pi}\sum_{k\in \mathbb{Z}\setminus\{0\}}\frac{1}{k} \sinc\left(\frac{\sigma z}{\pi}-k\right)=f-f \sinc\left(\frac{\sigma z}{\pi}\right),
$$
which  together with (\ref{samplingformulabeforeVT}) proves our statement.
Theorem is proven.

\end{proof}

\subsection{Weak recovery of bandlimited trajectories}

The next corollary is obvious consequence of Theorem \ref{V-T}.
\begin{corollary}\label{col}
If $D$ generates a one parameter $C_{0}$-group of isometries  and $f\in \mathbf{B}_{\sigma}(D)$ then for every functional $g^{*}\in {\bf E}^{*}$ the following sampling formula holds 
\begin{equation}\label{ws100}
\langle G_{D}(z)f, g^{*} \rangle=
\sinc\left(\frac{\sigma z}{\pi}\right)\langle f, g^{*}\rangle +
$$
$$
z\>\sinc\left(\frac{\sigma z}{\pi}\right)\left \langle Df , g^{*}\right\rangle+
\sum_{k\in \mathbb{Z}\setminus\{0\}}  
 \left(\frac{\sigma z}{k\pi}\right)    
  \sinc\left(\frac{\sigma z}{\pi}-k\right)
  \left \langle G_{D}\left(\frac{k\pi}{\sigma}\right)f , g^{*}\right\rangle,
\end{equation}
where the  series converges absolutely and  uniformly on  all strips of bounded width parallel to the real axis, and in particular, on compact subsets of $\mathbb{C}$.
\end{corollary}

The following statement can be found  in \cite{BSS-1} and in \cite{BRS}.

\begin{theorem}\label{weak sampling}
If $f\in B_{\tau }^{\infty}(\mathbb{R})$ for some $0\leq \tau<\sigma$,  then for $z\in \mathbb{C}$ one has
\begin{equation}\label{S2}
f(z)=\sum_{k\in \mathbb{Z}}  f\left(\frac{k\pi}{\sigma}\right) \sinc\left(\frac{\sigma}{\pi}z-k\right),
\end{equation}
where the series uniformly converges on compact subsets of $\mathbb{C}$.
\end{theorem}

\begin{remark} For  the function $f(z)\equiv 1\in B_{0 }^{\infty}(\mathbb{R})$ the series (\ref{S2}) 
takes the form

\begin{equation}\label{S3}
1=\sum_{k\in \mathbb{Z}}   \sinc\left(\frac{\sigma}{\pi}z-k\right).
\end{equation}
But this series does not converge absolutely if, for example, $\sigma=\pi$ and $z=0.5$.
Indeed, one has
$$
\sum_{k\in \mathbb{Z}} \left| \sinc(0.5-k)\right |=\sum_{k\in \mathbb{Z}}\frac{\left|\sin\left((0.5-k )\pi\right)\right|}{\left |\left(  (0.5-k)\pi\right)\right| }=\frac{1}{\pi}\sum_{k\in \mathbb{Z}}\frac{1}{|(0.5-k)|}=
$$
$$
\frac{2}{\pi}+\sum_{k\in \mathbb{Z}\setminus\{0\})}\frac{1}{|(0.5-k)|}=\frac{2}{\pi}+\sum_{k\in \mathbb{N}}\left( \frac{1}{k-0.5}+\frac{1}{k+0.5}\right)=
$$
$$
\frac{2}{\pi}+\sum_{k\in \mathbb{N}}\left( \frac{2k}{k^{2}-0.25}\right)\geq\frac{2}{\pi}+2\sum_{k\in \mathbb{N}}\frac{1}{k}>\infty.
$$
This example shows that in general the series (\ref{S2}) does not converge absolutely.

\end{remark}

Since for every $f\in {\bf B}_{\sigma}(D)$ and every $g^{*}\in {\bf E}^{*}\>$ the scalar function $\langle G_{D}(t)f,\>g^{*}\rangle=\langle e^{tD}f, g^{*}\rangle$ belongs to the regular space $B_{\sigma}^{\infty}(\mathbb{R})$, 
 Theorem \ref{weak sampling} implies our next sampling result. 
 
\begin{theorem}\label{ws200}

If $D$ generates a one-parameter $C_{0}$-group of isometries,  and $f_{0}\in {\bf B}_{\tau}(D),\>\>0\leq\tau<\sigma$, then for every $g^{*}\in {\bf E}^{*}\>$
\begin{equation}\label{ws200}
\left \langle G_{D}(z)f_{0},\>g^{*}\right\rangle=\sum_{k\in \mathbb{Z}}
\left\langle G_{D}\left (\frac{k\pi}{\sigma}\right)f_{0},\>g^{*}\right\rangle 
\sinc \left(\frac{\sigma}{\pi}z -k\right),\>\>z\in \mathbb{C},
\end{equation}
where the series converges uniformly on compact subsets of $\mathbb{C}$.
\end{theorem}

\section{Non-stationary equations in Banach spaces}

\subsection{The  abstract Schr\"{o}dinger-type equation in a Banach space ${\bf E}$}\label{Schrod-eq}

 Let  $D$ be a generator of one-parameter $C_{0}$-group $G_{D}(t), \>t\in \mathbb{R}$ of isometries. 
The corresponding    abstract Cauchy problem   associated with a  Schr\"{o}dinger-type abstract equation is the following.
 
 Find a vector valued function $ f:\mathbb{R}\mapsto  {\bf E},$ which satisfies the Schr\"{o}dinger equation
\begin{equation}\label{Seq00001}
\frac{d}{dt}f(t)=Df(t),\>\>
\end{equation}
and the initial condition
\begin{equation}\label{SC00001}
f(0)=f_{0}\in  \mathcal{D}(D),
\end{equation}
where 
the derivative with respect to $t$ is understood as
\begin{equation}\label{t-deriv}
\frac{d}{dt}f(t)=\lim_{\tau\rightarrow 0}\tau^{-1}\|f(t+\tau)-f(t)\|.
\end{equation}
The unique classical solution to this Cauchy problem is given by the formula $f(t)=G_{D}(t)f_{0}$.  In this section we will be interested in  Cauchy problems whose initial value belongs to a a corresponding Bernstein space, i.e.
\begin{equation}\label{Bernstein-initial-groups} 
f(0)=f_{0}\in {\bf B}_{\sigma}(D).
\end{equation}

It is shown in Theorem \ref{exp-solution} that when  $f_{0}\in         \mathbf{B}_{\sigma}(D)$ the trajectory $G_{D}(t)f_{0}: \mathbb{R}\mapsto {\bf B}_{\sigma}(D)$ is given by the formula
\begin{equation}
G_{D}(t)f_{0}=e^{itD}f_{0},
\end{equation}
and represents a vector valued function which has extension $G_{D}(z)f_{0}=e^{-izD}f_{0},\>z\in \mathbb{C}$ to $\mathbb{C}$ as entire function of exponential type $\sigma$ bounded on the real line.
For such bandlimited trajectories  the following theorem holds which is a consequence of Theorems \ref{V-T} and \ref{ws200} and Corollary \ref{col}.

\begin{theorem}\label{group-reconstruction}
If $D$ generates a one-parameter $C_{0}$-group of isometries  and $f\in \mathbf{B}_{\sigma}(D)$, 
then the unique solution of the Cauchy problem (\ref{Seq00001}), (\ref{Bernstein-initial-groups} ) has extension to $\mathbb{C}$ as an entire function $e^{zD}f_{0}$ of  exponential type $\sigma$. In this case the sampling and interpolation formulas (\ref{s100}), (\ref{ws100}), and (\ref{ws200}) hold.

\end{theorem}

\subsection{The abstract diffusion equation in a Banach space $\bf E$}\label{heat-eq}

Let $A$ be an operator  on $\EB$ that  generates  a  one parameter contraction $C_{0}$- semigroup 
$$
S_{A}: \mathbb{R}_{+}\mapsto \mathcal{B}(\EB).
$$ 
 It means that 
\begin{enumerate}
\item $S_{A}(t_{1}+t_{2})=S_{A}(t_{1})S_{A}(t_{2}),\>\>t_{1}, t_{2}\in \mathbb{R}_{+},$
\item $S_{A}(0)=I,\>\>(I - is \>the \>identity \>operator),$

\item  $\|S_{A}(t)\|\leq1,\>\>t\in \mathbb{R}_{+}$,
\item $\lim_{t\rightarrow 0+}\|S_{A}(t)-f\|=0,\>\>f\in \EB$.
\end{enumerate}

The Cauchy problem for the abstract  diffusion-type equation is formulated as follows: find a vector-valued function $h:\mathbb{R}_{+}\mapsto  {\bf E}$ which satisfies the equation 

\begin{equation}\label{Heq00001}
\frac{d}{d\tau}h(\tau)=Ah(\tau),\>\>
\end{equation}
and the initial condition 
\begin{equation}\label{HC00001}
h(0)=h_{0}\in  \mathcal{D}(A),                 
\end{equation}
where the derivative with respect to $\tau$ is understood as in (\ref{t-deriv}).

The unique classical solution to the Cauchy problem  (\ref{Heq00001}),  (\ref{HC00001}) is given by the formula $f(t)=S_{A}(t)f_{0}$.  
We will be interested in  Cauchy problems whose initial value belongs to a corresponding Bernstein space.
However, 
if $P $ generates in ${\bf E}$ a strongly continuous bounded semigroup $T_{P}$ (but not a group) then the set  $\bigcup_{\sigma\geq 0}\mathbf{B}_{\sigma}(P)$ may not be  dense in ${\bf E}$. 

\begin{counterexample} Consider a strongly continuous bounded semigroup $T_{P}(t)$ in $L_{2}(0,\infty)$ defined for every $f\in L_{2}(0,\infty)$ as $T_{P}(t)f(x)=f(x-t),$ if $\>x\geq t$ and $T_{P}(t)f(x)=0,$ if $\>\>0\leq x<t$.  Bernstein inequalities  imply  that if $ f\in \mathbf{B}_{\sigma}(P)$ then for any $g\in L_{2}(0,\infty)$ the function $\left<T_{P}(t)f,\>g\right>$ is analytic in $t$. Thus if $g$ has compact support then $\left<T_{P}(t)f,\>g\right>$ is zero for  all   $t\geq 0$, which implies that $f$ is zero. In other words, in this case every space $\mathbf{B}_{\sigma}(P)$ is trivial. 
\end{counterexample}

This example shows that if $A$ generates a semigroup and one wants  the set $\bigcup_{\sigma>0}{\bf B}_{\sigma}(A)$ to be dense in ${\bf E}$, then some additional restrictions on the operator $A$ must be imposed.  Our additional conditions are the following.
\begin{assumption}\label{assumption11}
We consider an operator $A$ which generates a contraction $C_{0}$-semigroup 
$S_{A}(t), \>\|S_{A}(t)\|\leq 1,\>0\leq t<\infty,$  and is  such that for some $\phi\in [ -\pi, \pi)$ the operator $e^{i\phi}A$ generates  a $C_{0}$-group of isometries $G_{e^{i\phi}A}(t),\>t\in \mathbb{R}$.  
\end{assumption}

Now, an analog of Definition \ref{Bernstein space} can be used to define ${\bf B}_{\sigma}(A)$. Since $e^{i\phi}A$ generates a group and 
$$
{\bf B}_{\sigma}(A)={\bf B}_{\sigma}(e^{i\phi}A),
$$
the union $\bigcup_{\sigma>0}{\bf B}_{\sigma}(A)$ is dense in ${\bf E}$.
In this section we will assume that 
\begin{equation}\label{Bernstein-initial-semigroups} 
h(0)=h_{0}\in {\bf B}_{\sigma}(A).
\end{equation}
The proof of the next lemma for semigroups is the same as the proof of Theorem \ref{exp-solution} for groups.

\begin{theorem}\label{exp-solution-semigroup}
Suppose that $A$ generates a one-parameter contraction $C_{0}$-semigroup $S_{A}$ with $ \|S_{A}(t)\|\leq 1, \>t\geq 0$, and  $h_{0}\in {\bf B}_{\sigma}(A)$.

Then the following properties hold.

 \begin{enumerate}

 \item The trajectory $S_{A}(t)h_{0},\>\>t\in \mathbb{R}_{+}\cup \{0\},$ has extension to the complex plain as an abstract valued entire function

\begin{equation}\label{ext}
S_{A}(z)h_{0}=\sum_{k=0}^{\infty}\frac{z^{k}}{k!}D^{k}h_{0}=e^{zD}h_{0},\>\>z\in \mathbb{C},\>\>h_{0}\in  {\bf B}_{\sigma}(A).
\end{equation}
where the series converges uniformly on compact subsets of $\mathbb{C}$.

\item The extended trajectory $S_{A}(z)h_{0}$ has exponential type, i.e.
$$
\|S_{A}(z)h_{0}\|\leq \|h_{0}\|e^{\sigma\left| z\right|},\>\>z\in \mathbb{C}.
$$

\item If $\Re\>z\geq0$ then
 \begin{equation}\label{Im}
\|S_{A}(z)h_{0}\|\leq \|h_{0}\|e^{\sigma\left|\Im\> z\right|}.
\end{equation}

\end{enumerate}

\end{theorem}

Since $A$ generates a bounded strongly continues semigroup  $S_{A}(t), \>t\geq 0,$ the unique classical solution to the problem \ref {Heq00001}, \ref{Bernstein-initial-semigroups}   is given by the formula 
$$
h(\tau)=S_{A}(\tau)h_{0},\>h_{0}\in \mathcal{D}(A),\>\tau\geq 0.
$$
 It follows from Lemma \ref{exp-solution-semigroup} that when  $h_{0}\in \mathbf{B}_{\sigma}(A)$ the trajectory $h(\tau),\>\tau\geq 0,$ is given by the formula
$$
S_{A}(\tau)h_{0}=\sum_{k}\frac{\tau^{k}}{k!}A^{k}h_{0},\>\>h_{0}\in \mathbf{B}_{\sigma}(A),\>\tau\geq 0,
$$
and represents a vector valued function which has extension to $\mathbb{C}$
\begin{equation}\label{exp-semigroup}
S_{A}(\xi)h_{0}=\sum_{k}\frac{\xi^{k}}{k!}A^{k}h_{0}=e^{\xi A}h_{0},\>\>h_{0}\in \mathbf{B}_{\sigma}(A),\>\xi \in \mathbb{C},
\end{equation}
 as an entire function of exponential type $\sigma$. One also has 
\begin{equation}\label{exp-group}
G_{e^{i\phi }A}(z)h_{0}=\sum_{k}\frac{  \left(ze^{i\phi}\right)^{k}  }{k!}A^{k}h_{0}=e^{ze^{i\phi}A}h_{0},\>\>h_{0}\in \mathbf{B}_{\sigma}(e^{i\phi}A),\>\>z\in \mathbb{C}.
\end{equation}
Comparing (\ref{exp-semigroup}) and (\ref{exp-group}) we conclude that when
\begin{equation}\label{Wick}
z=\xi e^{-i\phi},\>\>z,\xi \in \mathbb{C},
\end{equation}
then  for every $h_{0}\in \mathbf{B}_{\sigma}(A)=\mathbf{B}_{\sigma}(e^{i\phi}A)$ the following equalities hold 
$$
G_{e^{i\phi }A}(z)h_{0}=\sum_{k}\frac{z^{k}}{k!}e^{i k \phi}A^{k}h_{0}=\sum_{k}\frac{\xi^{k}}{k!}A^{k}h_{0}=S_{A}(\xi)h_{0}.
$$
Thus we established the following relations.
\begin{lemma}\label{group-semigroup}
If the assumption \ref{assumption11} holds , then for every $h_{0}\in \mathbf{B}_{\sigma}(A)=\mathbf{B}_{\sigma}(e^{i\phi}A)$ the following holds true.

\begin{enumerate}

\item The vector valued functions $G_{e^{i\phi }A}(t)h_{0},\>\>t\in \mathbb{R},$ and $S_{A}(\tau)h_{0},\>\>\tau\geq 0,$
have extensions to $\mathbb{C}$ as entire functions of exponential type $\sigma$.

\item The following formulas hold
$$
G_{e^{i\phi }A}(z)h_{0}= S_{A}\left(ze^{i\phi}\right)h_{0}, \>\>\>z\in \mathbb{C},
$$
or
$$
S_{A}(\xi )h_{0}=G_{e^{i\phi }A}(\xi e^{-i\phi})h_{0},\>\>\xi\in \mathbb{C}.
$$
\end{enumerate}
\end{lemma}

This fact  allows to obtain certain reconstruction formulas for solutions of the Cauchy problem \ref {Heq00001},\ref{Bernstein-initial-semigroups}  with bandlimited initial data. For example, consider the formula (\ref{s100}). For the  bounded group   $G_{e^{i\phi }A}(z)h_{0}$ with $h_{0}\in {\bf B}_{\omega}(e^{i\phi}A),\>\>0<\omega<\sigma,$  it will take the form 

\begin{equation}\label{group sampling 200}
G_{e^{i\psi}A}(z)h_{0}=\sinc\left(\frac{\sigma z}{\pi}\right)h_{0}+
$$
$$
z\>\sinc\left(\frac{\sigma z}{\pi}\right)e^{i\psi}Ah_{0} +
\sum_{k\neq 0}   \left(\frac{\sigma z}{k\pi}\right)     \sinc\left(\frac{\sigma z}{\pi}-k\right)G_{e^{i\psi}A}\left(\frac{k\pi}{\sigma}\right)h_{0} ,
\end{equation}
where the  series converges in the norm of ${\bf E}$ uniformly on every horizontal strip of finite width in $\mathbb{C}$. 

Thus for $z=\xi e^{-i\phi},\>\>\xi, z \in \mathbb{C},$ by using Lemma \ref{group-semigroup} we obtain a reconstruction formula for $\left \langle S_{A}(\tau)h_{0},\>g^{*}\right\rangle$
in terms of the samples of the group trajectory  $\left\langle G_{ e^{i\phi} A} \left(\frac{k\pi}{\sigma}\right)h_{0}, g^{*}\right\rangle$:

\begin{equation}\label{group sampling 2000}
S_{A}(\xi)h_{0}=\sinc\left(\frac{\sigma \xi e^{-i\phi}}{\pi}\right)h_{0}+
$$
$$
\xi \>\sinc\left(\frac{\sigma \xi e^{-i\phi}}{\pi}\right)Ah_{0} +
\sum_{k\neq 0}   \left(\frac{\sigma \xi e^{-i\phi}}{k\pi}\right)     \sinc\left(\frac{\sigma \xi e^{-i\phi}}{\pi}-k\right)G_{e^{i\psi}A}\left(\frac{k\pi}{\sigma}\right)h_{0} ,
\end{equation}
where the  series converges in the norm of ${\bf E}$ uniformly on every horizontal strip of finite width in $\mathbb{C}$. 

This formula can also be rewritten as 
\begin{equation}\label{group sampling 20000}
S_{A}(\xi)h_{0}=\sinc\left(\frac{\sigma \xi e^{-i\phi}}{\pi}\right)h_{0}+
$$
$$
\xi \>\sinc\left(\frac{\sigma \xi e^{-i\phi}}{\pi}\right)Ah_{0} +
\sum_{k\neq 0}   \left(\frac{\sigma \xi e^{-i\phi}}{k\pi}\right)     \sinc\left(\frac{\sigma \xi e^{-i\phi}}{\pi}-k\right)S_{A}\left(e^{i\phi}\frac{k\pi}{\sigma}\right)h_{0} ,
\end{equation}
where the  series converges in the norm of ${\bf E}$ uniformly on every horizontal strip of finite width. 
Thus we have the following theorem.

\begin{theorem}\label{semigroup reconstruction}
Let $A$ be an operator in ${\bf E}$ for which the Assumption \ref{assumption11} holds true.
Then for the unique solution of the Cauchy problem (\ref{Heq00001}), (\ref{Bernstein-initial-semigroups} ) 
 has extension to $\mathbb{C}$ as an entire function $e^{zD}h_{0}$ of  exponential type $\sigma$. Then the  sampling and interpolation formula (\ref{group sampling 2000}) holds. Moreover, 
the following sampling and interpolation formula for every $g^{*}\in {\bf E}^{*}$ holds true as well

\begin{equation}\label{s3}
\langle S_{A}(\xi)h_{0}, g^{*}\rangle=
\sinc\left(e^{-i\phi}\frac{\sigma \xi }{\pi}\right)\langle h_{0}, g^{*}\rangle+
\xi \>\sinc\left(e^{-i\phi}\frac{\sigma \xi}{\pi}\right) \langle Ah_{0}, g^{*}\rangle +
$$
$$
\sum_{k\in \mathbb{Z}\setminus\{0\}}   \left(e^{-i\phi}\frac{\sigma \xi}{k\pi}\right)     \sinc\left(e^{-i\phi}
\frac{\sigma \xi}{\pi}-k\right)\left \langle S_{A}\left (e^{i\phi}\frac{k\pi}{\sigma}\right)h_{0}, g^{*}\right\rangle,
\end{equation}
where convergence is absolute and uniform on compact subsets of  $\mathbb{C}$.

\item In addition, if   the initial value $h_{0}$ belongs to 
${\bf B}_{\tau}(D)$ where $0<\tau<\sigma$, then for any $g^{*}\in {\bf E}$ 

\begin{equation}\label{wsf-2}
\langle S_{A}(z)h_{0},\>g^{*}\rangle=\sum_{k\in \mathbb{Z}}\left\langle S_{R}\left(\frac{k\pi}{\sigma}\right)h_{0},\>g^{*}\right\rangle \sinc \left(\frac{\sigma}{\pi}z -k\right),\>\>z\in \mathbb{C},
\end{equation}
where the series converges uniformly on compact subsets of $\mathbb{C}$.

\end{theorem}

\subsection{ The  abstract wave-type equation} \label{wave-eq-1}

One considers a Cauchy problem for the abstract wave equation in which

\begin{enumerate}

\item 
 the operator $D$ generates a bounded group $G_{D}(t), \>t\in \mathbb{R}$, of class $C_{0}$ in a Banach space ${\bf E}$,

\item it has bounded inverse $D^{-1}$, which means existence of the continues  $D^{-1}$ and the relation $\mathcal{R}(D)={\bf E}$.

\end{enumerate}

\begin{equation}\label{Weq00001}
\frac{d^{2}}{ds^{2}}w(s)=D^{2}w(s),\>\>\>w:\mathbb{R}\mapsto  {\bf E},
\end{equation}
\begin{equation}\label{WC00001}
w(0)=w_{0}\in   \mathcal{D}(D^{2}),\>\>w^{'}(0)=w^{'}_{0} \in \mathcal{D}(D)\cap \mathcal{R}(D).     \>\>                    
\end{equation}
Since $D$ generates a  group of class $C_{0}$, the problem (\ref{Weq00001}), (\ref{WC00001}) has a unique solution \cite{Birman} (pp. 300-305), which is given by the formula
\begin{equation}\label{W2}
g(t)=\frac{1}{2}\left[G_{D}(t)+G_{D}(-t)\right]g_{0}+\frac{1}{2}\left[G_{D}(t)-G_{D}(-t)\right]D^{-1}g^{'}_{0}.
\end{equation}
Now we assume that
\begin{equation}\label{Wband00}
w_{0}\in {\bf B}_{\sigma}(D)\subset \mathcal{D}(D^{2}),\>\>\>w_{0}^{'}\in {\bf B}_{\sigma}(D)\cap D\left(  {\bf B}_{\sigma}(D) \right).
\end{equation}
The condition for $w_{0}^{'}$  implies  that the vector $D^{-1}w_{0}^{'}$ belongs to ${\bf B}_{\sigma}(D)$ and using Theorem 
\ref{group-reconstruction} we obtain the following result.

\begin{theorem}\label{wave-reconstruction}

If $D$ generates a one-parameter $C_{0}$-group of isometries  and has a bounded inverse,
then the unique solution $w(t)$ of the Cauchy problem (\ref{Weq00001}), (\ref{Wband00}) extends  to $\mathbb{C}$ as an entire function of exponential type $\sigma$. Moreover, the following sampling formulas hold.
\begin{enumerate}

\item  
\begin{equation}\label{wave-solution-reconstruction}
w(z)=
\sinc\left(\frac{\sigma z}{\pi}\right)g_{0}+
$$
$$
\frac{1}{2}\sum_{k\in \mathbb{Z}\setminus \{0\}}   \left(\frac{\sigma z}{k\pi}\right)  \left[      \sinc\left(\frac{\sigma z}{\pi}-k\right)-      \sinc\left(-\frac{\sigma z}{\pi}-k\right) \right]G_{D}\left(\frac{k\pi}{\sigma}\right)w_{0} +
$$
$$
z\>\sinc\left(\frac{\sigma z}{\pi}\right)D^{-1}w_{0}^{'}+
$$
$$
\frac{1}{2}\sum_{k\in \mathbb{Z}\setminus \{0\}}   \left(\frac{\sigma z}{k\pi}\right)  \left[      \sinc\left(\frac{\sigma z}{\pi}-k\right)+     \sinc\left(-\frac{\sigma z}{\pi}-k\right) \right]G_{D}\left(\frac{k\pi}{\sigma}\right)D^{-1}w_{0}^{'},
\end{equation}
where the  series converge in the norm of ${\bf E}$ uniformly on every horizontal strip of finite width.

\item For every $g^{*}\in {\bf E}^{*}$,

\begin{equation}\label{wave-solution-reconstruction}
\langle w(z), g^{*}\rangle=
\sinc\left(\frac{\sigma z}{\pi}\right)\langle w_{0}, g^{*}\rangle+
$$
$$
\frac{1}{2}\sum_{k\in \mathbb{Z}\setminus \{0\}}   \left(\frac{\sigma z}{k\pi}\right)  \left[      \sinc\left(\frac{\sigma z}{\pi}-k\right)-      \sinc\left(-\frac{\sigma z}{\pi}-k\right) \right]\left \langle G_{D}\left(\frac{k\pi}{\sigma}\right)w_{0}, g^{*}\right \rangle +
$$
$$
z\>\sinc\left(\frac{\sigma z}{\pi}\right)\left \langle D^{-1}w_{0}^{'}, g^{*}\right\rangle+
$$
$$
\frac{1}{2}\sum_{k\in \mathbb{Z}\setminus \{0\}}   \left(\frac{\sigma z}{k\pi}\right)  \left[      \sinc\left(\frac{\sigma z}{\pi}-k\right)+     \sinc\left(-\frac{\sigma z}{\pi}-k\right) \right]\left \langle G_{D}\left(\frac{k\pi}{\sigma}\right)D^{-1}w_{0}^{'}, g^{*}\right \rangle.
\end{equation}
where convergence is absolute and uniform on every horizontal strip of finite width.

\end{enumerate}

\end{theorem}

\section{Examples of the classical Cauchy problems} \label{classicalexamp}

\subsection{The  Cauchy problem  for the Schr\"{o}dinger equation.}
If $D$ is a self-adjoint operator in a Hilbert space, then by the Stone's Theorem the operator $iD$ generates a unitary group and the following is an abstract Schr\"{o}dinger equations and for the solutions of a Cauchy problem with initial conditions in ${\bf B}_{\sigma}(D)$ Theorem \ref{group-reconstruction} holds. 
The next is the Cauchy problem  for the classical Schr\"{o}dinger equation in which $D$ is the Laplacian $\Delta$.

\begin{equation}\label{eq0001}
\frac{d}{dt}\psi(t)=i\Delta\psi(t),\>\>i^{2}=-1,
\end{equation}
\begin{equation}\label{Cauchy1}
\psi(0)=\psi_{0},\>\>\>\psi_{0}\in W_{2}^{2}(\mathbb{R}^{d}),
\end{equation}
where $\psi: \mathbb{R} \mapsto W_{2}^{2}(\mathbb{R}^{d}).$ The spaces ${\bf B}_{\sigma}(\Delta),\>\sigma>0,$ coincide with the  Paley-Wiener spaces.

It is known that the propagator given by  the following convolution integral
$$
G_{i\Delta}(t)\psi_{0}(x)=\left(\frac{1}{4i\pi t}\right)^{d/2}\int_{\mathbb{R}^{d}}\exp\left(\frac{i|x-y|^{2}}{4t}\right)\psi_{0}(y)dy,\>\>t\in \mathbb{R}.
$$
In this case Theorem \ref{group-reconstruction} holds for  the group $G_{i\Delta}$. We also note that an example of the Cauchy problem for the Schr\"{o}dinger equation on combinatorial graphs was considered in \cite{Pes15a}.

\subsection{The Cauchy problem for the diffusion equation}  Let $A$ be a non-positive self-adjoint operator in a Hilbert space ${\bf H}$ which generates a uniformly bounded semigroup. Then the corresponding Cauchy problem for the abstract diffusion equation (abstract parabolic equation) has the form

\begin{equation}\label{eq1000}
\frac{d}{d\tau}\varphi(\tau)=A\varphi(\tau),
\end{equation}
\begin{equation}\label{Cauchy1000}
\varphi(0)=\varphi_{0},\>\>\>\varphi_{0}\in {\bf B}_{\sigma}(A),\>\sigma>0,
\end{equation}
where $\varphi: \mathbb{R}_{+} \mapsto \mathcal{D}(A),\> \tau\in \mathbb{R}_{+}$. The corresponding space ${\bf B}_{\sigma}(A),\>\sigma>0,$  denotes the image space of the projection operator ${\bf 1}_{[-\sigma, \sigma]}(A)$ (to be understood in the sense of Borel functional calculus). In this case Theorem \ref{semigroup reconstruction} holds.

We consider a Cauchy problem for the diffusion equation of the form 
\begin{equation}\label{eq100}
\frac{d}{d\tau}\varphi(\tau, x)=\Delta\varphi(\tau, x),
\end{equation}
\begin{equation}\label{Cauchy100}
\varphi(0,x)=\varphi_{0}(x),\>\>\>\varphi_{0}\in W^{2}_{2}(\mathbb{R}^{d}),
\end{equation}
where $\varphi: \mathbb{R} \mapsto W^{2}_{2}(\mathbb{R}^{d}),\> \tau\in [0, \infty), \>x\in \mathbb{R}^{d}, $ and $\Delta$ is the Laplacian in $L_{2}(\mathbb{R}^{d})$ whose domain is the Sobolev space $W^{2}_{2}(\mathbb{R}^{d})$. The spaces ${\bf B}_{\sigma}(\Delta)$ are just the regular Paley-Wiener spaces.

It is know that the Cauchy problem (\ref{eq100}), (\ref{Cauchy100}) is uniformly correct  in the sense of Hadamard on the half-line $[0, \infty)$ and its unique solution is given by the formula
\begin{equation}
\psi(t, x)=S_{\Delta}(t)\varphi_{0}(x),\>\>\>t\in \mathbb{R},
\end{equation}
where $S_{\Delta}(t), \>\>t\in \mathbb{R}$ is the unitary group of operators generated by the operator $\Delta$. 
Here \cite{BB}
$$
S_{\Delta}(t)\varphi(x)=(4\pi t)^{-d/2}\int_{\mathbb{R}^{d}}\varphi(x-\xi)\exp\left(\frac{-|\xi|^{2}}{4t}\right)dt,\>\>0<t<\infty,
$$
and $S_{\Delta}(0)\varphi(x)=\varphi(x)$. In this situation Theorem \ref{semigroup reconstruction} holds for the semigroup $S_{\Delta}$.

\subsection{ The  Cauchy problem for the wave  equation} \label{wave-eq}

The operator $\Delta-\epsilon I$, where $\epsilon >0$ and $ I $ is the identity operator, is a strictly negative definite self-adjoint operator on $L_{2}(\mathbb{R}^{d})$. If $\sqrt{ -\Delta+\epsilon I}$ is the unique positive square root of  the positive definite operator  $ -\Delta+\epsilon I$, then $i\sqrt{ -\Delta+\epsilon I}$ generates a unitary group $W_{i \sqrt{ -\Delta+\epsilon I}}$ and
$$
\left[i\left(\sqrt{ -\Delta+\epsilon I}\right)\right]^{2}=\Delta-\epsilon I.
$$
Clearly, the operator $i \sqrt{ -\Delta+\epsilon I}$ has a bounded inverse.
We consider the Cauchy problem in $L_{2}(\mathbb{R}^{d})$ 
\begin{equation}\label{Weq001}
\frac{d^{2}}{ds^{2}}w(s)=\left(\Delta -\epsilon I\right)w(s),\>\>\>w:\mathbb{R}\mapsto  {\bf E},\>\>\epsilon>0,
\end{equation}
\begin{equation}\label{Wband}
w_{0}\in {\bf B}_{\sigma}(\sqrt{-\Delta+\epsilon I}),\>\>w_{0}^{'}\in {\bf B}_{\sigma}(\sqrt{-\Delta+\epsilon I})\cap\sqrt{-\Delta+\epsilon I}\left(  {\bf B}_{\sigma}(\sqrt{-\Delta+\epsilon I} \right). 
\end{equation}
In this situation  Theorem \ref{wave-reconstruction}  holds true, where the  group $W_{D}$  is  replaced by the unitary group $W_{i\sqrt{-\Delta+\epsilon I} }$, which  can be written as 
$$
W_{i\sqrt{-\Delta+\epsilon I}}(t)w_{0}(x)=\frac{1}{(2\pi)^{d}}\int_{\mathbb{R}^{d}}
e^{i( x \cdot  \xi +t \sqrt{|\xi |^{2}+\epsilon}) }\mathcal{F}(w_{0})(\xi)d\xi,\>\>x,\xi \in \mathbb{R}^{d},
 $$
 where $\mathcal{F}$ is the Fourier transform.

\section{Appendix. 
Density of the abstract Bernstein spaces}\label{density}

\begin{theorem}\label{Density}
If $D$ generates a one parameter group of isometries $T_{D}$ of class $C_{0}$ in
a Banach space ${\bf E}$ then the set  $\bigcup_{\sigma\geq 0}\mathbf{B}_{\sigma}(D)$ is dense in $E$.
\end{theorem}

\begin{proof}

Note that if  $\phi\in L_{1}(\mathbb{R}),\>\>\>\|\phi\|_{1}=1,$ is an entire function of exponential
type $\sigma$ then for any $f\in {\bf E}$ the vector
$$
g=\int _{-\infty}^{\infty}\phi(t)T_{D}(t)fdt
$$
belongs to $\mathbf{B}_{\sigma}(D).$ Indeed,  for every real $\tau$ we
have
$$
T_{D}(\tau)g=\int_{-\infty}^{\infty}\phi(t)T_{D}(t+\tau)Dfdt=\int_{-\infty}^{\infty}\phi(
t-\tau)T_{D}(t)fdt.
$$
 Using this formula we can extend the abstract function $T_{D}(\tau )g$ to the
complex plane as
$$
T_{D}(z)g=\int_{-\infty}^{\infty}\phi(t-z)T_{D}(t)fdt.
$$
 Since by the assumption  $h$ is an entire function of exponential
type $\sigma$ and $\|\phi\|_{L_{1}(\mathbb{R})}=1$
 we have
$$
\|T_{D}(z)g\|\leq
\|f\|\int_{-\infty}^{\infty}|\phi(t-z)|dt\leq\|f\|e^{\sigma|z|}.
 $$
 This inequality  implies that $g$ belongs to $\mathbf{B}_{\sigma}(D)$. 
Let
 $$
h(t)=a\left(\frac{\sin (t/4)}{t}\right)^{4}
$$
 and
$$
a=\left(\int_{-\infty}^{\infty}\left(\frac{\sin
(t/4)}{t}\right)^{4}dt\right)^{-1}.
$$
Function $h$ will have the
 following properties:

\begin{enumerate}

\item $h$ is an even nonnegative entire function of exponential
type one;

\item $h$ belongs to $L_{1}(\mathbb{R})$ and its
$L_{1}(\mathbb{R})$-norm is $1$;

\item  the integral
\begin{equation}\int_{-\infty}^{\infty}h(t)|t|dt
\end{equation} 
is finite.
\end{enumerate}
Consider
 the following vector 
\begin{equation}
\mathcal{R}_{h}^{\sigma}(f)=
\int_{-\infty}^{\infty}h(t)T_{D}(t/\sigma)
 fdt=\int_{-\infty}^{\infty}h(t\sigma)T_{D}(t)f dt,
 \end{equation}
Since the function $h(t)$ has exponential type one,  the function
$h(t\sigma)$ has the type $\sigma$. It  implies  that $\mathcal{R}_{h}^{\sigma}(f)$ belongs to $\mathbf{B}_{\sigma}(D)$.

The modulus of
continuity is defined as in \cite{BB}
$$
\Omega(f,s)=\sup_{|\tau|\leq
s}\left\|\Delta_{\tau}f\right\|,\>\>\>\>
 \Delta_{\tau}f=(I-T_{D}(\tau))f. 
$$
Note, that for every $f\in {\bf E}$ the 
modulus $\Omega(f,s)$ goes to zero when $s$ goes to zero.  Below we are using  an easy verifiable inequality
$
\Omega\left(f, as\right)\leq \left(1+a\right)\Omega(f,
s),\>\>\> a\in \mathbb{R}_{+}.
$
We obtain
$$
\|f-\mathcal{R}_{h}^{\sigma}(f)\|\leq
\int_{-\infty}^{\infty}h(t)\left\|\Delta_{t/\sigma}f\right\|dt\leq
\int_{-\infty}^{\infty}h(t)\Omega\left(f, t/\sigma\right)dt\leq 
$$
$$
\Omega\left(f,
\sigma^{-1}\right)\int_{-\infty}^{\infty}h(t)(1+|t|)dt\leq
{C}_{h}\Omega\left(f,
\sigma^{-1}\right),
$$
where the integral
 $$
C_{h}=\int_{-\infty}^{\infty}h(t)(1+|t|)dt
$$
 is finite by the choice of $h$. Theorem  is proven.
 
\end{proof}

\end{document}